\documentclass[english, 12pt]{article}
\usepackage{amsfonts}
\usepackage{amsmath}
\usepackage{amssymb}
\usepackage{amsthm}
\usepackage{amscd}
\usepackage{amscd}
\usepackage{amsthm}
\numberwithin{equation}{section}

\usepackage{amsfonts}
\usepackage{amsmath}
\usepackage{color}
\usepackage{mathrsfs}
\usepackage{dsfont}
\usepackage{enumerate} 
\newtheorem{theo}{Theorem}[section]
\newtheorem{prop}[theo]{Proposition}
\newtheorem{lem}[theo]{Lemma}

\newtheorem{coro}[theo]{Corollary}
\theoremstyle{definition}
\newtheorem{defi}[theo]{Definition}
\newtheorem{rk}[theo]{Remark}

\newtheorem{exa}[theo]{Example}

\newtheorem*{coro*}{Corollary}

\def\scrL{\tilde{\mathscr{F}}}

\def\R{\mathbb{R}}
\def\N{\mathbb{N}}
\def\CC{\mathbb{C}}

\newcommand{\En}{\mathscr{E}}
\newcommand{\IR}{\mathbb{R}}

\newcommand{\F}{\mathscr{F}}

\newcommand{\tete}{(T_t;t\ge 0)}

\DeclareMathOperator{\dom}{Dom}

\newcommand{\m}{m}

\newcommand{\RR}{\mathds{R}}
\newcommand{\NN}{\mathds{N}}

\newcommand{\PP}{\mathds{P}}
\newcommand{\EE}{\mathds{E}}
\newcommand{\One}{\mathbf{1}}

\newcommand{\eex}{\mathds{E}^x}
\newcommand{\set}[2]{\bigl\{#1\bigm|#2\bigr\}}
\begin{document}
\title{Essential spectrum and Feller type properties}
\author{\normalsize
Ali BenAmor\footnote{High school for transport and logistics, University of Sousse, Tunisia, email: ali.benamor@ipeit.rnu.tn}, Batu G\"uneysu \footnote{Faculty of Mathematics, TU Chemnitz, email: batu.gueneysu@math.tu-chemnitz.de
} and Peter Stollmann  \footnote{Faculty of Mathematics, TU Chemnitz, email: stollman@math.tu-chemnitz.de
}}
\date{22.02.2022}

\maketitle
\noindent 
\begin{abstract}
We give necessary and sufficient conditions for a regular semi-Dirichlet form to enjoy a new Feller type property, which we call \emph{weak Feller property}. Our characterization involves potential theoretic as well as probabilistic aspects and seems to be new even in the symmetric case. As a consequence, in the symmetric case, we obtain a new variant of a decomposition principle of the essential spectrum for (the self-adjoint operators induced by) regular symmetric Dirichlet forms and a Persson type theorem, which applies e.g. to Cheeger forms on $\mathsf{RCD^*}$ spaces. 
\end{abstract}
\section{Introduction}

Let $X$ be a locally compact separable metrizable space equipped with a positive Radon measure $m$ with full support. Let $\En$ be a regular semi-Dirichlet form on $L^2(X)$ with $H$ the associated sectorial operator and $T_t:=e^{-tH}$ the associated semigroup. Given a subset $A\subset X$ we denote by $\mathrm{cap}(A)$ the induced capacity, and given an open subset $U\subset X$ we denote by $H_U$ the restriction of $H$ to $U$, which is a sectorial operator in $L^2(U)\subset L^2(X)$. We refer the reader to Section \ref{preliminary} for detailed definitions. The starting point for our investigations is the following result from \cite{LS}:

\begin{theo}\label{alt} Assume $\En$ is symmetric and spatially locally compact, in the sense that $\One_A T_t$ is compact for all $t>0$ and all Borel sets $A\subset X$ with $m(A)<\infty$. Then for all open $U\subset X$ with $\mathrm{cap}(X\setminus U)<\infty$ one has $\sigma_\mathrm{ess}(H_U)=\sigma_\mathrm{ess}(H)$.
 \end{theo}	

Although the class of $U$'s considered in this result is very large (in fact, somewhat optimal), the restriction on the underlying geometry through the spatial local compactness assumption is rather strong: for example, there exist $\mathsf{RCD}^*$ spaces of finite measure such that the (operator induced by the) Cheeger form (cf. Example \ref{ex2} for the definitions) does not have a purely discrete spectrum and thus is not spatially locally compact. Nontrivial examples of such spaces are provided by the natural Dirichlet forms (cf. Example \ref{ex1}) of certain complete Riemannian manifolds with Ricci curvature bounded from below: indeed, noncompact hyperbolic manifolds with finite volume are never spatially locally compact (as these always have an nonempty essential spectrum \cite{mazzeo}). The main goal of this paper is to obtain a variant of Theorem $\ref{alt}$, which allows to treat geometries such as arbitrary $\mathsf{RCD}^*$ spaces. To this end, we found a surprising connection between such results and probability theory. Namely, we say that $\En$ has the \emph{weak Feller property}, if the following two properties are satisfied:
\begin{itemize}
	\item[($\alpha)$] with $L^\infty_0(X)$ the space of all $u\in L^\infty(X)$ such that for all $\epsilon>0$ there exists $K\subset X$ compact such that $\|1_{X\setminus K}u \|_\infty <\epsilon$, one has 
\[
T_t(L^\infty_0(X))\subset L^\infty_0(X)\quad\text{for all $t>0$.}
\]
\item[$(\beta)$] For any compact  $K\subset X$  there exists a function $0\leq w\in L^2(X)\cap L^\infty_0(X)$ with $(H+1)^{-1 } w \ge \One_K $. 
\end{itemize}
Recall that $\En$ is said to \emph{induce a Feller semigroup}, if
\begin{align}\label{fel}
	&T_t(C_0(X))\subset C_0(X)\quad\text{for all $t>0$},\\
	&\| T_t\phi -\phi\|_\infty\to 0\quad\text{  as  $t\to 0+$  for all $\phi\in C_0(X)$},
\end{align}
with $C_0(X)$ the space functions on $X$ that vanish at $\infty$. In general, this property implies the weak Feller property and these notions are equivalent on Riemannian manifolds, but there exist $\En$'s which have the weak Feller property but not the Feller property (cf. Section \ref{diffusionvshitting} and Section \ref{ex}). We provide a list of equivalent characterizations of the weak Feller property under the above condition ($\alpha$), one of which is the traditional probabilistic condition that compact sets are hard to hit from close to infinity by the underlying diffusion: namely, whenever $K$ is compact and $\sigma_K$ denotes the induced first hitting time, then for all $t\geq 0$ one has  
\begin{equation}\label{hardtohit}
	\PP^\bullet\{ \sigma_K\le t\}\in L^\infty_0(X),
\end{equation}
keeping in mind that by the seminal work of Azencott \cite{A} it is well-known that \emph{in the Riemannian case}, (\ref{hardtohit}) is \emph{equivalent} to the Feller property.\\
With this definition, one of our main results is:

\begin{theo}\label{intro1} Let $\En$ be symmetric and satisfy the weak Feller property, and let $\En$ be \emph{weakly} spatially local compact, in the sense that $\One_KT_t$ is compact for all compact $K\subset X$ and all $t>0$. Then for every open $U\subset X$ with $X\setminus U$ compact one has $\sigma_{\mathrm{ess}} (H_U) = \sigma_{\mathrm{ess}} (H)$.
\end{theo}
The weak Feller property enters the proof of Theorem \ref{intro1}  precisely in the form (\ref{hardtohit}), through a method from \cite{LS}.\\
Concerning the assumptions of Theorem \ref{intro1} we remark that the natural Dirichlet form on a connected Riemannian manifold is automatically weakly spatially locally compact and satisfies the condition ($\beta$), and has the (weak) Feller property, e.g., if the manifold is complete and its Ricci curvature does not decay to fast to $-\infty$; likewise, the Cheeger energy of an arbitrary $\mathsf{RCD}^*$ space as well as many jump diffusion Dirichlet forms satisfy all assumptions of the above theorem (cf. Example \ref{ex1}, Example \ref{ex2}, Example \ref{ex3} and Section \ref{decomposition}). \\
Finally, we obtain a variant of Persson's theorem, which states that under the same assumptions on $\En$ as in Theorem \ref{intro1} one has 
	$$
	\inf \sigma_{\mathrm{ess}} (H) =\lim_{K\to X} \inf \sigma (H_{X \setminus K}),
	$$
and we show that these assumptions on $\En$ are satisfied, if $\En$ has the doubly Feller property in the sense of \cite{KKT}, that is, if $\En$ has the Feller property in addition to $T_t(L^\infty(X))\subset C_0(X)$ for all $t>0$.\vspace{1mm}

For completeness, we have also included an appendix, where we show that every semi-Dirichlet form whose semigroup satisfies (\ref{fel}) is automatically regular, a result that comes in handy e.g. for possibly nonsymmetric jump diffusions.

\section{Preliminaries}\label{preliminary}

For the standard terminology concerning Dirichlet forms we refer to \cite{FOT} for the symmetric and to \cite{MR,Oshima} for the non-symmetric case. In the sequel, $\left\|\cdot \right\|_p$ for $p\in [1,\infty]$ denotes the norm on $L^p(X)$. Following \cite{Oshima}, we say that a pair $(\En,\F)$ (or shortly $\En$, if there is no danger of confusion) is a \emph{semi-Dirichlet form on $L^2(X)$}, if the following assumptions are satisfied:
\begin{itemize}
 \item $\F$ is a dense subspace of $L^2(X)$ and $\En:\F\times\F\to\RR$ is bilinear and $\En[u]:=\En(u,u)\ge 0$.
 \item Sector condition: there exists a constant $K\ge 1$ such that
$$
|\En(u,v)|\le K\En(u,u)^{\frac{1}{2}}\En(v,v)^{\frac{1}{2}} .
$$
\item $\F$ is complete w.r.t. $\|\cdot\|_{\En}$, where 
$$
\| u \|_{\En}^2:= \En[u]+\| u\|^2.
$$
\item Markovian property: for all $u\in\F$, $a\ge  0$, one has 
\begin{equation}\label{markovian}
	\quad u\wedge a \in \F\mbox{ and }\En(u \wedge a, u - u\wedge a)\ge 0.
\end{equation}
\end{itemize}

In the above situation, we get a maximally sectorial operator $H$ on $L^2(X)$ associated with $\En$ by Kato's first form representation theorem, \cite[VI,Thm 2.1, p. 322]{Kato} and for $t\geq 0$, $\alpha>0$ we write
$$
T_t:=e^{-tH}, \quad G_\alpha:=(H+\alpha)^{-1}
$$
for the corresponding contraction semigroup and resolvent, respectively, see \cite[Thm 1.1.2, p. 4]{Oshima}. As a consequence of \eqref{markovian}, the semigroup $\tete$ is a sub-Markov semigroup \cite[Thm 1.1.5, p. 7]{Oshima} and so extends in a $p$-consistent way to a contraction semigroup $L^p(X)$ for all $p\in [1,\infty]$ which is strongly continuous for $p<\infty$ and weak-*-continuous for $p=\infty$. Likewise, $G_\alpha$ is bounded in $L^p(X)$ for all $p\in [1,\infty]$. To ease notation, we do not distinguish between these semigroups and the resolvents just write $T_t$ resp. $G_\alpha$ for all of them. \vspace{1mm}

The semi-Dirichlet form $(\En,\F)$ is called \emph{regular}, if $\F\cap C_c(X)$ is dense in $\F$ with respect to $\|\cdot\|_{\En}$ and in $C_c(X)$ with respect to $\left\|\cdot\right\|_\infty$.\vspace{1mm}

We fix once for all a regular semi-Dirichlet form $(\En,\F)$ on $L^2(X)$.\vspace{2mm}

Then $\En$ is called \emph{symmetric}, if one has $\En(u,v)=\En(v,u)$ for all $u,v\in \F$, noting that a regular symmetric semi-Dirichlet in our sense is automatically a regular symmetric Dirichlet form in the standard sense of Fukushima \cite{Fukushima}. Moreover, $\En$ is called \emph{strongly local}, if $\En(u,v)=0$, whenever $u,v\in\F$ are such that $u$ is constant on the support of $v$.\vspace{1mm}

We set
$$
\En_\alpha(u,v):=\En(u,v)+\alpha\int_X uv \ dm,\quad \En_\alpha(u):=\En_\alpha(u,u), \quad\text{for all $\alpha>0$},
$$
and remark that the induced Choquet capacity is defined on open subsets $U\subset X$ by
$$
\mathrm{cap}(U):=\inf \set{\En_1(v)}{v\in\F, \One_U\leq v},
$$
with the usual convention $\inf \emptyset=\infty$, and for arbitrary $A\subset X$ one then sets
$$
\mathrm{cap}(A):= \inf \set{\mathrm{cap}(U)}{ A\subset U, U\mbox{  open}}.
$$
With respect to this capacity, every $u\in \F$ has a (quasi-)unique quasi-continuous representative $\tilde{u}$, and keeping this in mind, for every set $B$ one sets
$$
\scrL_{B,1}:= \set{u\in\F}{\tilde{u}\geq \One_B \>\>\text{q.e.}},
$$
so that for open $U\subset X$ one has
$$
\mathrm{cap}(U)=\inf \set{\En_1(v)}{ v\in  \scrL_{U,1}}.
$$
The regularity of $\En$ implies (through the existence of cut-off functions) that 
$$
\mathrm{cap}(K)<\infty\quad\text{for all compact $K\subset X$.}
$$

We will also be concerned with the Hunt process 
$$
\mathscr{M}:=(\Omega,(\PP^x;x\in X),(X_t;t\ge 0))
$$
associated with $\En$ (see \cite{FOT,MR},\cite[Section 3.3]{Oshima} and the groundbreaking \cite{Fukushima}), with lifetime $\zeta\in (0,\infty]$; it givesthe following probabilistic representation for the semigroup: for all $t>0$, $f\in L^2(X)$ and $m$-a.e. $x\in X$ one has
$$
T_tf(x)=\EE^x\{ \One_{\{t<\zeta\}} f(X_t) \} .
$$

The restriction of $\En$ to $\overline{ \F \cap C_c (U)}^{\|\cdot\|_{\En}}$ is a regular semi-Dirichlet form on $L^2 (U)\subset L^2(X)$, \cite[Section 3.5, in particular Thm 3.5.7]{Oshima}. This form will be denoted by $\En_U$. The maximally sectorial operator
associated to $\En_U$ will be denoted by $H_U$ and its semigroup with $(T^U_t;t\geq 0)$.\vspace{1mm}

The \emph{first exit time} of $\mathscr{M}$ from a Borel set $B$ is defined by
$$
\tau_B:= \inf \set{t> 0}{X_t\not\in B}
$$
and the \emph{first hitting time} of $\mathscr{M}$ of $B$ is defined by
$$
\sigma_B:= \tau_{X\setminus B}=\inf \set{t> 0}{X_t \in B} .
$$
The following form of the Feynman-Kac formula allows a probabilistic interpretation of the semigroup generated by $H_U$: for all $f\in L^2(U)$, $t> 0$, $m$-a.e. $x\in U$ one has
\begin{equation}\label{Feynman-Kac}
T^U_tf(x) =\eex\left[\One_{\{t<\tau_U\}} f( X_t) \right]
\end{equation}
(see \cite[Theorem 3.5.7, p. 100 and its proof]{Oshima}, see also \cite{CF}, Section 3.3, in particular Thm 3.3.8, p. 109f. for the symmetric case). We will view bounded operators in $L^p(U)$ to be acting on $L^p(X)$, by defining them to be $0$ on $L^p(X\setminus U)$.\vspace{1mm}

We finally introduce some potential theoretic notions that will be needed in the sequel: let $u\in L^p(X)$ for some $1\le p\le \infty$. We say that $u$ is \emph{$1$-excessive}, if $e^{-t}T_t u \leq u$ for all $t>0$.\vspace{1mm}

The \emph{$1$-equilibrium potential} $e_B$ for a Borel set $B$ is defined by 
\begin{equation}
e_B (x) = \eex(e^{-\sigma_B}) .
\label{potential-expectation}
\end{equation}
Actually, this is a little shortcut that is convenient for what we have in mind. Typically, $e_B$ is introduced by a variational principle, and \eqref{potential-expectation} is then deduced as an important link between the stochastic and the analytic world. In this spirit, the RHS of the definition of $e_B$ is typically denoted as $p^1_B(x)$ or $H_B1(x)$, see \cite[Lemma 3.1.1]{CF} or \cite[Lemma 3.4.3]{Oshima}.

\begin{rk}\label{potential}\begin{itemize}
           \item [(1)] In the symmetric case, for any Borel set $B$ such that $\scrL_{B,1}\not=\emptyset$, the function $e_B$ is the unique function in $\scrL_{B,1}$ which satisfies
$$
\En_1(e_B,e_B)=\min \set{\En_1(u)}{u\in \scrL_{B,1}},
$$
and one then has
$$
\mathrm{cap}(B)=\En_1(e_B,e_B),
$$
see \cite[p. 78, p.105]{CF}.
\item [(2)] In the non--symmetric case, at least for open $B$, \cite[Lemma 3.4.3]{Oshima}
gives an analogous statement, however, the variational problem is somewhat different then.
\item [(3)]  \cite[Proposition 2.8 (iii)]{MOR} gives, that for all open $B$ and all $1$-excessive $u\in\scrL_{B,1}$ one has $e_B\le u$, another variational property of $e_B$ that will be of prime importance later.
\item [(4)] If $X$ is a connected Riemannian manifold with $m$ its volume measure and $\En$ is given by 
$$
\En(u,v)=\int_X (\nabla u,\nabla v) dm\quad\text{with domain of definition $\F:= W^{1,2}_0(X)$,}
$$
then for any open relatively compact $U\subset X$ the function $e_U$ is the minimal nonnegative solution of the exterior boundary value problem
\begin{equation}\label{ddssdd}
 \begin{aligned}
  \Delta u &= u\mbox{ in }X\setminus \overline{U}\\
  u&>0 \mbox{  on  }U^c\\
  u&=1 \mbox{  on }\partial U.
 \end{aligned}
\end{equation}
\item [(5)] Let $U\subset X$ be open and relatively compact. Then there is a unique positive $\sigma$-finite Borel measure $\mu_U$ on $X$, supported in $\overline U$, such that $e_U$ is the $1$-potential of $\mu_U$, that is,
\[
\mathscr{E}_1(e_B,u) = \int_X u\,d\mu_U\quad\text{ for all $u\in\F$},
\]
and one has $\mathrm{cap}(U) = \mu_U(\overline U)$. If $G_1$ has an integral kernel $\kappa_1(x,y)$, then one has 
\begin{equation}
e_U(x) = \int_X \kappa_1(x,y)\,d\mu_U(y).
\label{PotRepresentation}
\end{equation}
Indeed, let $\kappa_1^* (x,y):=\kappa_1 (y,x)$ and let $G_1^*$ be the dual resolvent of $G_1$ ($G_1^*$ is in fact the resolvent of the dual form $\mathscr{E}^*(u,v)= \mathscr{E}(v,u)$). Obviously $\kappa_1^*$ is the kernel of $G_1^*$. Let $u\in L^2(X), u\geq 0$. Then
\begin{align*}
	\mathscr{E}_1( e_U, G_1^*u) & = \int_X G_1^*u(x)\,d\mu_U(x) = \mathscr{E}_1^*( G_1^*u, e_U) = \int_X u(y) e_U(y)\,dm(y)\\ 
	& = \int_X\int_X \kappa_1^*(x,y)u(y)\,dm(y)\,d\mu_U(x)\\
	&=\int_X u(y)\cdot(  \int_X \kappa_1^*(x,y)\,d\mu_U(x))\,dm(y),
\end{align*}
and the claim follows, by extending the latter identity to every $u\in L^2(X)$ through $u = u^+ - u^-$.

\end{itemize}
\end{rk}
\section{The weak Feller property}\label{diffusionvshitting}
Recall that by definition $\En$ induces a Feller semigroup, if one has (\ref{fel}). The generalization of this property that we seek for relies on the space
\[
L^\infty_0(X):=\set{ u\in L^\infty(X)}{\text{for all}\ \epsilon>0,\,\exists\ K\subset X\ \text{compact s.t.}\ 
\|u 1_{X\setminus K}\|_\infty <\epsilon }
\]
of bounded functions that vanish at $\infty$ in the measure theoretic sense.
\begin{defi}($\alpha$)
 We say that $\En$ satisfies the $L^\infty_0$--\emph{diffusion property}, provided
\begin{equation}
 T_t(L^\infty_0(X))\subset L^\infty_0(X)\label{l0diffusion} .
\end{equation}
 ($\beta$) We say that $\En$ satisfies property \textbf{(P)}, if for any compact  $K\subset X$  there is $w\in L^2(X)\cap L^\infty_0(X), w\ge 0$  such that $G_1 w \ge \One_K$.

($\gamma$) We say that $\En$ satisfies the \emph{weak Feller property}, if it has the $L^\infty_0$--diffusion property and if \textbf{(P)} holds.
\end{defi}

We discuss a number of sufficient conditions:

\begin{rk}\label{remarkpositivity}
 \begin{itemize}
  \item [\rm{(1)}] The following property implies $\rm{\textbf{(P)}}$: \\
  \textbf{(P'')}:  for any compact  $K\subset X$  there exists $w\in L^2(X)\cap L^\infty_0(X), w\ge 0$ and $t'>0$ such that for all $0<t<t'$ one has $T_t w \ge \One_K$, which follows from 
  $$
  G_1 =\int_0^\infty e^{-t}T_t dt,
  $$
  where the improper Riemannian integral is well-defined in the norm topology of $L^p(X)$ if $p\in [1,\infty)$, and in the weak-*-topology of $L^\infty(X)$.

  \item [\rm{(2)}] If $\En$ induces a Feller semigroup, then it also has the weak Feller property: indeed, given a compact $K\subset X$ pick $0\leq w\in C_c(X)$ with $w\geq c$ in $K$ for some $c>0$. Then, since $t\mapsto T_t$ is strongly continuous at $t=0$ in $C_0(X)$, we can pick a $c'>0$ and $t'>0$ such that $T_tw\geq c'$ in $K$ for all $0<t<t'$, so that the claim follows from the previous item.
  
  \item [\rm{(3)}] Property \rm{\textbf{(P)}} is equivalent to:\\
\textbf{(P')}: there is $w\in L^2(X)\cap L^\infty_0(X),w\ge 0$ such that for any compact $K\subset X$ one has
\begin{equation}\label{positivity}
G_1 w\ge \One_K .
\end{equation}
\item [\rm{(4)}] By local compactness and the semigroup property, the $L^\infty_0$--diffusion property is equivalent to the following property:
\begin{equation*}
T_t\One_K\in  L^\infty_0(X)\quad\text{ for all compact $K\subset X$, $0<t<1$.}
\end{equation*}
 \end{itemize}
\end{rk}

We are now in position to give a characterization of the weak Feller property:
\begin{theo}\label{main2}
Under assumption \textbf{(P)}, the following assertions are equivalent:
\begin{itemize}
  \item [\rm{(1)}]  $\En$ satisfies the $L^\infty_0$--diffusion property.
  \item [\rm{(2)}]  For any compact $K\subset X$, $t\ge 0$ one has $\PP^{\bullet}\{\sigma_K\le t\}\in  L^\infty_0(X).$
 \item [\rm{(3)}]  For any compact $K\subset X$ one has $e_K\in  L^\infty_0(X).$
 \item [\rm{(4)}]  For any compact $K\subset X$ there exists an $1$-excessive function $\phi\in\mathscr{F}$ with $\One_K\le \phi\in L^\infty_0(X).$
\end{itemize}	
\label{Weak-Feller}
\end{theo}
\begin{proof}(4)$\Rightarrow$(3): Let $K$ as in (3) and pick a relatively compact,
open $B\supset K$. By (4) there exists an $1$-excessive $\phi\in L^\infty_0(X)\cap\mathscr{F}$ such that
$\One_B\le\phi$. From Remark \ref{potential},(3) above, we get that $e_B\le\phi$ and so
$$
e_K\le e_B\in L^\infty_0(X) .
$$
(3)$\Rightarrow$(2): follows from 
$$
\PP^{\bullet}\{\sigma_K\le t\}\le e^t e_K.$$
(2)$\Rightarrow$(1): follows from Remark \ref{remarkpositivity}, (4) and
$$
T_t\One_K=\EE^{\bullet}[\One_K\circ X_t]\le \PP^{\bullet}\{\sigma_K\le t\} .
$$
(1)$\Rightarrow$(4): Pick $w$ as in \textbf{(P)} and set
\begin{equation}\label{psi}
\phi:= G_1w =\int_0^\infty e^{-t}T_tw dt.
\end{equation}
Condition \textbf{(P)} ensures that $\phi\ge \One_K$ and evidently, $\phi$ is $1$-excessive and belongs to  $\mathscr{F}$. Finally, by the $L^\infty_0$--diffusion property we know that $T_t\phi \in L^\infty_0(X)$ for $t\ge 0$ and since the integral in \eqref{psi} converges w.r.t the uniform norm, it follows that $\phi\in L^\infty_0(X)$ as asserted.
\end{proof}
Although in most concrete applications (see the next section) the weak Feller property is most conveniently checked in terms of the semigroup, the above result allows in principle to check this property directly through the equilibrium potential. Assume for example that $X$ is a metric space with metric $d(x,y)$, that $G_1$ has an integral kernel $\kappa_1(x,y)$ and that for every open relatively compact $U\subset X$ the function $G_1\One_U$ is bounded from below by a strictly positive lower semicontinuous function (so that one has property $(\textbf{P}))$. If then for any such $U$ there exist constants $c, \delta, \gamma>0$ such that
\begin{equation}
\kappa_1(x,y) \leq c d(x,y)^{-\gamma}\ \text{for all $y\in U$ and all $x\in X$ with $ d(x,y)>\delta$},
\label{GreenUB}
\end{equation}
then one has $e_U\in L_0^\infty(X)$, which clearly entails the weak Feller property. Indeed, let $n\in\N$ be large enough and let $x\in U$ be such that $d(x,\overline U)>n$. From formula (\ref{PotRepresentation}) we obtain
\begin{align*}
e_U(x) & = \int_X \kappa_1(x,y)\,d\mu_U(y) =\int_{\overline U} \kappa_1(x,y)\,d\mu_U(y) 
\leq c \int_{\overline U} d(x,y)^{-\gamma}\,d\mu_U(y)\\
& \leq \frac{c}{n^\gamma} \mu_U(\overline U) = \frac{c}{n^\gamma} \mathrm{cap}(U),
\end{align*}
which letting $n\to\infty$ proves the claim. Situations as above can be inferred for example from results such as \cite[Section 5.6]{LSCoste}) \cite[Lemma 7.7]{GH} \cite[Lemma 2.4]{CGL}.

\section{Examples of weakly Feller $\En's$}\label{ex}

In this section we give some classes of examples to illustrate the weak Feller property.

\begin{exa}\textbf{(Multiplication operators)} Let $h:X\to [0,\infty)$ be measurable and define a regular symmetric Dirichlet form by
 \begin{align*}
  &\F=\set{f\in L^2(X)}{h^{\frac12}f \in L^2(X)}\\
  &\En(f,g)=\int_X h\cdot f \cdot g \ dm.
   \end{align*}
Then $H$ is the maximally defined multiplication operator induced by $h$ and $T_t$ is the bounded multiplication operator induced by $e^{-th}$.

\begin{enumerate}[(i)]
 \item The $L^\infty_0$--diffusion property always holds, in view of $|T_t f|\le |f|$ for all $t\ge 0$.
  \item Property \textbf{(P)} is equivalent to $h\in L^\infty_{\mathrm{loc}}(X)$.
 \item The Feller property is equivalent to $h\in C(X)$, keeping in my mind that the required strong continuity in (\ref{fel}) automatically follows from Dini's theorem.
 \item One has $T_t(C_0(X))\subset C_0(X)$  for all $t\ge 0$, if and only if $T_t(C_b(X))\subset C_b(X)$ for all $t\geq 0$, which in turn is equivalent to $h\in C(X)$, as mentioned already.
\end{enumerate}
\end{exa}

\begin{exa}\label{ex1}\textbf{(Riemannian manifolds)} Let $X$ be a connected Riemannian manifold with $m$ its volume measure, $d(x,y)$ the geodesic distance and $\En$ is the regular strongly local symmetric Dirichlet form in $L^2(X)$ given by 
	\begin{align}\label{riem}
	\En(u,v)=\int_X (\nabla u,\nabla v) dm\quad\text{with domain of definition $\F:= W^{1,2}_0(X)$.}
	\end{align}
Then $H$ is the Friedrichs realization of the Laplace-Beltrami operator $-\Delta$ and $T_t$ has a jointly smooth integral kernel $p_t(x,y)$ for all $t>0$, which is strictly positive by the parabolic maximum principle. Accordingly, $\En$ satisfies property $\textbf{(P'')}$, as given $K\subset X$ compact one has $T_tw\geq \One_K$
 for 
 $$
 w:=\big(\m(K)\inf_{x,y\in K}p_t(x,y)\big)^{-1}\One_K .
$$
 Moreover, $\En$ induces a Feller semigroup, if and only if one has $T_t(C_0(X))\subset C_0(X)$: indeed, the required strong continuity follows from writing  
$$
P_tf-f=\int^t_0P_s\Delta f ds\quad\text{for all $t>0$, $f\in C_c^\infty(X)$,}
$$
and using that $P_s$ is a contraction in $L^\infty(X)$ We have borrowed this argument from the proof of Proposition 4.3 in \cite{pallara}. Moreover, $\En$ induces a Feller semigroup, if and only if \cite[Theorem 2.2]{PS} the unique minimal solution to (\ref{ddssdd}) vanishes at $\infty$, for all open relatively compact $U\subset X$ (see also \cite[Theorem 3.3]{WO} for analogous result on graphs). Since by Remark \ref{potential}(4) this solution is precisely $e_U$, it follows from Theorem \ref{Weak-Feller} that $\En$ is weakly Feller, if and only if it has the  Feller property. This is the case, e.g., if $X$ is geodesically complete and its Ricci curvature satisfies \cite{hsu}
$$
\mathrm{Ric}\geq -C_1-C_2d(x,y)^2\quad\text{for some $C_1,C_2\geq 0$.}
$$	
\end{exa}

\begin{exa}\label{ex2}\textbf{($\mathsf{RCD}^*$ spaces)} Let $X$ be a complete separable geodesic metric measure space with metric $d(x,y)$ and let $m$ be an inner regular Borel measure on $X$ with full support, which is finite on all open balls $B(x,r)$. The \emph{Cheeger form} 
	$$
	\mathrm{Ch}:L^2(X)\longrightarrow [0,\infty]
	$$
	on $X$ is defined to be the $L^2$-lower semicontinuous relaxation of the functional 
	$$
	\widetilde{\mathrm{Ch}}:L^2(X)\cap \mathrm{Lip}(X)\longrightarrow [0,\infty]
	$$
	given by 
	$$
	\widetilde{\mathrm{Ch}}(f):= \int_X   \Big( \limsup_{y\to x}\frac{|f(x)-f(y)|}{d(x,y)} \Big) ^2 \ dm(x),
	$$
	One gets a functional $\En$ on $L^2(X)$ with domain of definition $\F$ given by all $f$ with $\mathrm{Ch}(f)<\infty$ by setting $\En(f):=\mathrm{Ch}(f)$ for such $f$'s. The metric measure space $X$ is called \emph{infinitesimally Hilbertian}, if $\En$ is a quadratic form. For example, if $X$ is a complete connected Riemannian manifold with its geodesic distance and Riemannian volume measure, then $\En$ is given by (\ref{riem}). 
	
	Next, let us give the following definition of the $\mathsf{CD}^*(K,N)$ condition from \cite{SB}, which is a slight generalization of the original $\mathsf{CD}(K,N)$ condition from \cite{sturm}\cite{lott} having the advantage of admiting a natural local-to-global principle: let $\mathsf{P}(X)$ denote the set of all Borel probability measures on $X$ and let $\mathsf{P}_2(X)$ denote the elements of $\mathsf{P}(X)$ that have finite second moments. Given $\mu_0,\mu_1\in\mathsf{P}_2(X)$, the $L^2$-\emph{Wasserstein distance} is defined by
	$$ 
	W_2(\mu_0,\mu_1):=\inf\left\{\int_{X\times X}d(x,y)^2 dq(x,y):\>\>q\in\mathscr{C}(\mu_0,\mu_1)\right\},
	$$
	with $\mathscr{C}(\mu_0,\mu_1)$ the set of all couplings between $\mu_0$ and $\mu_1$. Any minimizer of the above infimum is called an \emph{optimal coupling} of $\mu_0$ and $\mu_1$. Then $\mathsf{P}_2(X)$ together with the $L^2$-Wasserstein distance is again a complete separable geodesic space (as $X$ is so). Given $K\in\mathbb{R}$, $t\in [0,1]$, $N\in [1,\infty)$ define the function
	$$
	\sigma_{K,N}^{(t)}:[0,\infty)\longrightarrow \IR\cup \{\infty\}
	$$
	by
	\begin{align*}
	\sigma_{K,N}^{(t)}(\theta):=
	\begin{cases}
		&\infty, \quad\text{if $K\theta^2\geq N\pi^2$}\\
		&\frac{\sin(t\theta\sqrt{K/N})}{\sin(\theta\sqrt{K/N})},\quad \text{if $0<K\theta^2< N\pi^2$}\\
		&t, \quad\text{if $K\theta^2=0$}\\
		&\frac{\sinh(t\theta\sqrt{-K/N})}{\sinh(\theta\sqrt{-K/N})}, \quad\text{if $K\theta^2< 0$}.
	\end{cases}
	\end{align*}
	Given $K\in\IR$, $N\in [1,\infty)$, the metric measure space $X$ is called an \emph{$\mathsf{CD}^*(K,N)$ space}, if for all $\mu_0,\mu_1\in \mathsf{P}(X)$ with bounded support and $\mu_0,\mu_1\ll\m$, there exists an optimal coupling $q$ of them and a geodesic $(\mu_t)_{t\in [0,1]}$ in $\mathsf{P}_2(X)$ connecting them, such that
		$$ 
		\text{$\mu_t\ll\m$ and $\mu_t$ has a bounded support for all $t\in [0,1]$},
		$$
		and such that for all $N'\in[N,\infty)$, $t\in [0,1]$ one has
			\begin{align*}
				&\int_X \rho_t(x)^{1-1/N'} \m(dx)\geq \\
				&\int_{X\times X} \Big(\sigma_{K,N'}^{(1-t)}(\mathfrak{d}(x_0,x_1))\rho_0(x_0)^{-1/N'} +\sigma_{K,N'}^{(t)}(\mathfrak{d}(x_0,x_1))\rho_1(x_1)^{-1/N'}\Big)\>dq(x_0, x_1), 
			\end{align*}
			where $\rho_t$ denotes the Radon-Nikodym density of $\mu_t$ with respect to $\m$ for all $t\in [0,1]$. Finally, given $K\in\IR$, $N\in [1,\infty)$, the metric measure space $X$ is called an \emph{$\mathsf{RCD}^*(K,N)$ space}, if it is an infinitesimally Hilbertian $\mathsf{CD}^*(K,N)$ space. This definition is shown to be equivalent to weak Bochner type inequality in \cite{EKS}. Examples of such space include complete connected Riemannian manifolds of dimension $n\leq N$ and Ricci curvature bounded from below by $K$ with their geodesic distance and volume measure, but also possibly very singular spaces such as Alexandrov spaces of dimension $n\leq N$ with curvature $\geq K/(n-1)$ and their $n$-dimensional Hausdorff measure. \vspace{1mm}

			We fix an $\mathsf{RCD}^*(K,N)$ space $X$ in the sequel. \vspace{1mm}
			
			Then, by the validity of the Bishop-Gromov volume estimate \cite{SB}, there exists $C>0$ such that for all $0<r\leq R$ and all $x\in X$ one has the volume local doubling
			
			\begin{equation}\label{eq:LVD}
				\frac{\m(x,R)}{\m(x,r)}\leq C \mathrm{e}^{C R} (R/r)^N,
			\end{equation}
		which implies that $X$ is locally compact. Here, we have set $\m(y,s):=\m(B(y,s))$. Moreover, by a standard argument one gets the following local volume comparison inequality from the latter estimate: there is a constant $C'>0$ such that for all $t>0$, $x_1,x_2\in M$ and $\varepsilon>0$,
			
			\begin{equation}\label{eq:VC}
				\frac{\m(x_2,\sqrt{t})}{\m(x_1,\sqrt{t})}\leq C' \mathrm{e}^{\frac{C't}{{\epsilon}}} \mathrm{e}^{\epsilon\frac{\varrho(x_1,x_2)^2}{t}}.
			\end{equation}
As shown in \cite{KK}, $\En$ becomes a regular strongly local symmetric Dirichlet form. The semigroup $T_t$ has a jointly continuous integral kernel $p_t(x,y)$ for all $t>0$, which satisfies \cite{jiang} for all $x,y\in X$, $0<t<1$ the following Gaussian upper bounds
		$$
		C_1 m(x,\sqrt{t})^{-1} e^{-\frac{d(x,y)^2}{C_2t}}p_t(x,y) \leq C_3 m(x,\sqrt{t})^{-1} e^{-\frac{d(x,y)^2}{C_4t}}.
		$$
Clearly, this implies $p_t(x,y)>0$, so that property $\mathbf{(P'')}$ is satisfied as in the Riemannian case. For the proof of the $L^\infty_0$-diffusion property we can assume that $X$ is noncompact. We are going to show that $T_t1_K\in L^\infty_0(X)$ for all $0<t<1$. Indeed, the Gaussian upper bound implies
$$
|T_t\One_K(x)|\leq C_5C_3  \int_K e^{\frac{-d(x,y)^2}{C_4t}} dm(y),
$$
which clearly proves the claim, noting that 
$$
C_5:=\sup_{y\in K }m(y,\sqrt{t})^{-1}<\infty,
$$
as for every fixed $z\in X\setminus K$, by local volume comparison, one has 
$$
\inf_{y\in K}\m(y,\sqrt{t})\geq \inf_{y\in K} C'e^{-C't} e^{-d(z,y)^2/t} \m(z,\sqrt{t})>0.
$$
\end{exa}

While the previous two examples were both strongly local and symmetric, we finally provide an example which is nonlocal and nonsymmetric:

\begin{exa}\label{ex3}\textbf{(Nonsymmetric jump diffusions)} Assume $X=\IR^n$ with $m$ its Lebesgue measure, and fix $0<\alpha<2$ and a Borel function 
	\begin{align*}
	&\kappa:\IR^n\times \IR^n\longrightarrow  \IR,\quad\text{with}\\
	&\kappa_0\leq \kappa(x,z)\leq \kappa_1,\quad \kappa(x,z)=\kappa(x,-z),\\
	&|\kappa(x,z)-\kappa(y,z)|\leq \kappa_2|x-y|^\beta,
	\quad\text{for all $x,z,y\in\IR^n$,}
	\end{align*}
	for some constants $\kappa_0,\kappa_1,\kappa_2>0$, $0<\beta<1$. Define the operator $\tilde{H}$ on functions $f:\IR^m\to\IR$ by
	\begin{align}\label{stable}
	\tilde{H} f(x):=\lim_{\epsilon\to 0}\int_{\{|z|>\epsilon\}}(f(x+z)+f(z))\frac{\kappa(x,z)}{|z|^{m+\alpha}} dz,
	\end{align}
	whenever the expression makes sense.\\
	Let $p_t^\alpha$ be the heat kernelof the fractional Laplacian on $\R^d$. Then it follows from Theorem 1.1 in \cite{chen} (and a scaling argument) that there exists a unique jointly continuous function 
	$$
	p:(0,\infty)\times \IR^m\times \IR^m\longrightarrow [0,\infty),\quad (t,x,y)\longmapsto p_t(x,y)
	$$
	which has the following three properties:
	\begin{itemize}
		\item[i)] one has 
	$$
	\partial_tp_t(x,y)=	\tilde{H}p_t(\bullet,y)(x),\quad\text{for all $t>0$, $x\neq y$,}
	$$
	\item[ii)] for all $t_0>0$ there exists a constant $c_1>0$ such that for all $0<t<t_0$ and all $x,y$ one has
	$$
	p_t(x,y)\leq c_1t(t^{1/\alpha}+|x-y|)^{-m-\alpha},
		$$
	\item[iii)] for every $0<\gamma<\min(\alpha,1)$ and $t_0>0$ there exists a constant $c_2>0$ such that for all $x,x',y\in\IR^m$, $0<t<t_0$ one has
	$$
	|p_t(x,y)-p_t(x',y)|\leq c_1 |x-x'|^{\gamma}t^{1-\gamma/\alpha}\left(t^{1/\alpha}+\min(|x-y|,|x'-y|)\right)^{-m-\alpha},
	$$
	\item[iv)] the map $t\mapsto \tilde{H}p_t(\cdot,y)(x)$ is continuous for all $x\neq y$, and for all $t_0>0$ there exists a constant $c_3>0$ such that for all $0<t<t_0$ and all $x\neq y$ one has
	$$
	| \tilde{H}p_t(\bullet,y)(x)|\leq c_3(t^{1/\alpha}+|x-y|)^{-m-\alpha},
	$$
	\item[v)] for all bounded uniformly continuous functions $f:\IR^m\to\IR$ one has 
	$$
	\lim_{t\to 0+}\left\| \int_{\IR^m} p_t(\bullet,y)f(y)dy -f\right\|_\infty =0.
	$$
	\end{itemize}
	Moreover, $(t,x,y)\mapsto p_t(x,y)$ satisfies the Chapman-Kolmogorov identities, and with
	$$
	T_tf(x):=\int_{\IR^m} p_t(x,y)f(y)dy,
	$$
	the semigroup $(T_t;t\geq 0)$ is analytic in $L^2(\IR^m)$, and one has the conservativeness
	$$
	\int_{\IR^m} p_t(x,y)dy= 1\quad\text{for all $t>0$, $x,y\in\IR^m$.}
	$$
Assuming that the function $\kappa^*(x,y):=\kappa(y,x)$ satisfies the the same assumptions as $\kappa$, then with $H$ the generator of this semigroup, we can define \cite{MR} a semi-Dirichlet form $\En$ in $L^2(\IR^m)$ by defining $\F$ to be the completion of $\dom(H)$ with respect to the norm  
	$$
	u\longmapsto \sqrt{\left\langle Hu,u\right\rangle +\left\|u \right\|_2},
	$$
	and defining $\En$ to be the unique bilinear extension of 
	$$
	(u,v)\longmapsto \left\langle Hu,v\right\rangle 
	$$
	to $\F$. It follows from the continuity of the heat kernel $p_t(x,y)$, ii), v) and Theorem \ref{appendix} that $\En$ is regular and induces a Feller semigroup (and so a weak Feller semigroup). Moreover, $\En$ is symmetric, if and only if $\kappa$ is so. This example generalizes the usual $\alpha$-stable Dirichlet form in $\IR^m$, where $\kappa$ is a constant. Even more general nonlocal semi-Dirchlet forms have been treated along the same lines in \cite{KSV}, where the function $z\mapsto |z|^{-m-\alpha}$ in (\ref{stable}) has been replaced by a general class of functions subject to certain growth conditions.
		
\end{exa}

\section{The decomposition principle}\label{decomposition}

Here, we state some spectral theoretic consequences of Feller type properties, mainly addressing the issue of stability of the essential spectrum. Since functional calculus will be involved, \vspace{2mm}

we fix a regular symmetric Dirichlet form $\En$ throughout this section, \vspace{2mm}

with $H$, $T_t$ denoting, as usual, the associated self adjoint operator and semigroup and $\En_U$, $H_U$, $T^U_t$ the corresponding restriction to an open $U\subset X$, as discussed in Section \ref{preliminary}. We denote by $\left\| \cdot \right\|$ the operator norm in $L^2(X)$. 

\noindent
Our main aim is to show that
$$
\sigma_{\mathrm{ess}}(H)=\sigma_{\mathrm{ess}}(H_{X\setminus K})
$$
provided $K$ is compact and $H$ is suitable. The essential spectrum of a self-adjoint operator $S$ can be characterized as
$$
\sigma_{\mathrm{ess}}(S)=\set{\lambda \in \CC}{S-\lambda \mbox{  is not a Fredholm operator}},
$$
see \cite[Chapter XIII.14]{ReedSimon4}. It was introduced by H. Weyl in a slightly different form in \cite{weyl10}, who also proved its essential property, namely that it is invariant under compact perturbations, see \cite{weyl09}.

\begin{defi}
 We say that $\En$ satisfies the \emph{weak spatial local compactness property}, if 
 $\One_K T_t$ is a compact operator for every compact $K\subset X$ and every $t>0$.
\end{defi}

Clearly, this property is weaker than spatial local compactness, introduced in \cite[Definition 2.1]{LS} which requires that $\One_A T_t$ is a compact operator for every Borel set $A\subset X$ with $m(A)<\infty$.

The main point is the following variant of the corresponding result in \cite{LS}:

\begin{lem}\label{cor-limit}
 Let $\En$ satisfy the weak Feller property.
 Let $K\subset X$ be compact. Then there is a
 sequence $(K_n)_{n\in\NN}$ of compact sets such that for every $t>0$
 \begin{equation}\label{limit}
 \lim_{n\to\infty}\left\| \One_{K_n}\left( T_t -T_t^{X\setminus K}  \right)- \left(T_t -T_t^{X\setminus K}\right)\right\| =0.
 \end{equation}
\end{lem}
This result is proved exactly as  \cite[Corollary 2.5.]{LS}, where
$$
K_n\supset \{e_K>1/n\}
$$
can be chosen compact in view of the $L^\infty_0$--diffusion property of $\En$, for $n$ large enough. With exacly the same argument as in \cite{LS}, one can now prove:

\begin{theo}[Decomposition principle]\label{theorem-decomposition-principle}
Let $\En$ satisfy the weak Feller property and $H$ the weak spatial local compactness property.
Let $K\subset X$ be compact. Then  the operator
 $\varphi (H_{X\setminus K}) - \varphi (H)$ is compact for every $\varphi \in C_0 (\RR)$. In
particular,
$$\sigma_{\mathrm{ess}} (H_{X\setminus K}) = \sigma_{\mathrm{ess}} (H) .$$
\end{theo}

\begin{rk} The $\En$'s from Example \ref{ex1}, Example \ref{ex2} and Example \ref{ex3} (assuming symmetry in the last case) are weakly locally compact: indeed, in each case $T_t$ has a jointly continuous integral kernel, which shows that the Hilbert-Schmidt norm of the integral operator $\One_K T_t$ is finite for all compact $K\subset X$, $t>0$. Thus, for $\mathsf{RCD}^*$-spaces and symmetric jump diffusions Theorem \ref{theorem-decomposition-principle} applies directly. In the Riemannian case, one has to guarantee the Feller property, which holds, as noted above, if the manifold is complete with a Ricci curvature that does not decay to fast to $\infty$. In fact, as shown in \cite{donnelly} with completely different methods, the decomposition principle holds in the Riemannian case without any assumptions on the geometry. From this point of view, the main strenght of Theorem \ref{theorem-decomposition-principle} stems from the fact that it can deal with many local and nonlocal situations \emph{simultaniously}.
\end{rk}

We get the following variant of the Persson type theorem from \cite{LS}, 
where $\lim_{K\to X}$ stands for the limit along the net of 
compact subsets of $X$: 

\begin{theo}[Persson's theorem]\label{persson}
Let $\En$ satisfy the weak Feller property and $H$ the weak spatial local compactness property. Then,  
$$\inf \sigma_{\mathrm{ess}} (H) =\lim_{K\to X} \inf \sigma (H_{X \setminus K}).$$
\end{theo}

Finally, we present further results that connect Feller type properties with the stability of the essential spectrum. The first one shows that the compactness of $\varphi (H_{X\setminus K}) - \varphi (H)$ deduced in Theorem \ref{theorem-decomposition-principle} actually implies the corresponding spatial local compactness property:

\begin{prop}\label{prop5.5}
 Let $B\subset X$ be closed, $U:=X \setminus B$. If the operator
 $\varphi (H_U) - \varphi (H)$ is compact for every $\varphi \in C_0 (\RR)$ then  $\One_B T_t$ is a compact operator as well.
\end{prop}
We have chosen the assumption and the conclusion so as to fit the above results, many different equivalent formulations are at hand:
\begin{rk} Let $U\subset X$ be open and $B\subset X$ be closed.\\
1. The following properties are equivalent:
  \begin{itemize}
   \item $\varphi (H_U) - \varphi (H)$ is compact for every $\varphi \in C_0 (\RR)$,
   \item $T_t^U - T_t$ is compact for one (every) $t>0$,
   \item $D_\lambda:= G_\lambda - ( H_U + \lambda)^{-1} $ is compact for one (every) $\lambda>0$.
  \end{itemize}
2. The following properties are equivalent:
  \begin{itemize}
   \item $\One_B\varphi (H)$ is compact for every $\varphi \in C_0 (\RR)$,
   \item $\One_B T_t$ is compact for one (every) $t>0$,
   \item $\One_B G_\lambda $ is compact for one (every) $\lambda>0$.
  \end{itemize}
\end{rk}
For the proof of the second statement, see \cite[Thm 1.3]{LSW}; the first statement can be deduced with the help of the Stone-Weierstrass theorem, see  proof of Thm. 2.3 in \cite{LS}. 
\begin{proof}[Proof of Prop. \ref{prop5.5}]
 By \cite[Lemma 3]{BB}
\[
D_1:= G_1 - ( H_U + 1)^{-1} = (\check{H}^{1/2} J G_1)^* \check{H}^{1/2} J G_1,
\] 
where
\[
J:(\F,\mathscr{E}_1)\to L^2(B,\One_B m),\ u\mapsto u|_B,
\]
and $\check{H}$ is the selfadjoint operator of the trace of $\mathscr{E}_1$ w.r.t. the map $J$. Accordingly, $D_1$ is compact if and only if $\check{H}^{1/2} J G_1$ is compact as well. Since $\check{H}\ge 1$, we get that $J G_1$ is compact, whence the assertion in view of the preceding remark, part (2). 
\end{proof}
\begin{prop}\label{prop5.7}
 Let $\En$ satisfy the Feller property. If  $T_1^V$ is compact for some relatively compact open $V$, then $\One_K T_1$ is compact for every compact subset $K\subset V$. Consequently, if $T_1^V$ is compact for all relatively compact open $V$, then $H$ satisfies the weak spatial local compactness property. 
\end{prop}
\begin{proof}
 Let $K$ and $V$ be as in the assertion. By \cite[Lemma 2.2]{KKT} we get:
 \begin{equation*}
 \label{taube}
 \sup_{x\in K}\PP^x\left\{\tau_V\le s\right\}\to 0\mbox{  as  }s\searrow 0 .
 \end{equation*}
With an argument just as in \cite[Proof of Proposition 2.4]{LS} this gives that
\begin{equation}
 \label{norm}
 \| \One_K (T_s - T^V_s)\|\to 0\mbox{  as  }s\searrow 0 .
 \end{equation}
Since $T_1^V$ is compact, $T^V_s$ is compact as well by Remark \ref{prop5.5}, (1), applied to $H^V$ and $B=V$, we finally conclude that
\begin{align*}
 \One_K T_1 &=\One_K T_s T_{1-s}\mbox{  for  }s<1\\
 &=\lim_{s\searrow 0}\One_K T^V_s T_{1-s}
\end{align*}
is compact, as a norm limit (see \eqref{norm}) of compact operators.

The second assertion is evident.
\end{proof}
Using a compactness result \cite[Corollary 4.1]{KKT} for doubly Feller forms, we get:

\begin{coro}
 Let $\En$ be \emph{doubly Feller}, i.e., $\En$ induces a semigroup and one has the smoothing property $T_t(L^\infty(X))\subset C_b(X)$ for every $t>0$. Then $H$ satisfies the weak spatial local compactness property. In particular, the conclusions of Theorem \ref{theorem-decomposition-principle} and Theorem \ref{persson} hold.
\end{coro}

\section{Appendix}

Let again $X$ be a locally compact compact separable metrizable space which is equipped with a positive Radon measure $m$ with full support. We fix a semi-Dirichlet form $(\En,\F)$ on $L^2(X)$, with $H$ the associated sectorial operator, $T_t:=e^{-tH}$ the semigroup, and $G_\alpha$ the resolvent, both considered to be acting as contractions in $L^p(X)$. Finally, we are going to need the norm
$$
\left\| u\right\|_\En := \sqrt{\En(u,u)+\left\| u\right\|^2_2}\quad\text{on $\F$} 
$$ 
and the norm 
$$
\left\| u\right\|_H := \sqrt{\left\| Hu\right\|^2_2+\left\| u\right\|^2_2} \quad\text{on $\dom(H)$.}
$$

Our goal here is to prove the following result, which should be known to experts:

\begin{theo}\label{appendix} Assume one has 
	\begin{align}\label{fel1}
		&T_t(C_0(X))\subset C_0(X)\quad\text{for all $t>0$},\\
	\label{fel2}	&\| T_t\phi -\phi\|_\infty\to 0\quad\text{as $t\to 0+$  for all $\phi\in C_0(X)$}.
	\end{align}
Then $\En$ is regular.
\end{theo}

The proof relies on the following auxiliary results:

\begin{prop} The following statements are equivalent:
\begin{itemize}
	\item $\En$ is regular.
	\item $\F\cap C_0(X)$ is dense in $(\F,\left\| \bullet\right\|_\En)$ and in $(C_0(X),\left\| \bullet\right\|_\infty)$. 
	
\end{itemize}	

\end{prop} 

\begin{proof} It suffices to show that every $\phi\in \F\cap C_0(X)$ can be approximated by a sequence in $\F\cap C_c(X)$ with respect to $\left\| \bullet\right\|_\En$. To this end, by the Markovian property of $\En$, we can assume $\phi\geq 0$. Then one has $\phi_n:=(\phi-1/n)_+\in \F\cap C_c(X)$ by the Markovian property, $\phi_n\to \phi$ in $L^2(X)$, and moreover $\sup_n \En(\phi_n,\phi_n)<\infty$. Thus, a subsequence of $\phi_n$ does the job.
\end{proof}

\begin{prop} Under (\ref{fel1}) and (\ref{fel2}), the space $D_0:=\dom(H)\cap C_0(X)$ is dense in $(\dom(H),\left\| \bullet\right\|_H )$ and in $(C_0(X),\left\| \bullet\right\|_\infty)$.
\end{prop}

\begin{proof} Note first that in this situation $(T_t;t\geq 0)$ becomes a strongly continuous contraction semigroup in $(C_0(X),\left\| \bullet\right\|_\infty)$. Thus, one has 
	\begin{align}\label{lap}
		G_\alpha \phi= \int^\infty_0 e^{-\alpha t}T_t \phi \ dt\quad\text{for every $\alpha>0$, $\phi\in C_0(X)$,}
	\end{align}
	where the integral converges in the uniform norm.\\
1. With 
$$
D_1:=\set{G_\alpha\phi}{\phi\in C_c(X),\alpha>0},
$$
in view of (\ref{lap}), one has $D_1\subset D_0$. Given $g\in \dom(H)$ we have $f:=G_1g\in L^2(X)$, so there exists a sequence $f_n$ in $C_c(X)$ with $f_n\to f$ in $L^2(X)$. It follows that with $g_n:=G_1f_n\in D_1$ we have
$$
 g_n\to g\quad\text{in $(\dom(H),\left\| \bullet\right\|_H )$,}
$$
showing that $D_1$ is dense in $(\dom(H),\left\| \bullet\right\|_H )$.\\
2. Given $\phi\in C_c(X)$ we have $\alpha G_\alpha \phi\in D_1$ for all $\alpha>0$ and 
$$
 \alpha G_\alpha \phi\to \phi\quad\text{in $(C_0(X),\left\| \bullet\right\|_\infty)$, as $\alpha\to\infty$,}
$$ 
in view of formula (\ref{lap}): indeed, given $\epsilon>0$ we can pick $\delta>0$ such that for all $t\in [0,\delta]$ one has $\left\| T_t\phi-\phi\right\|_\infty<\epsilon/2$. It then follows easily from decomposing the integral in (\ref{lap}) as $\int^\infty_0\cdots=\int^\delta_0\dots+\int^\infty_\delta\cdots$ that for all $\alpha>0$ one has 
$$
\left\| \alpha G_\alpha\phi-\phi\right\|_\infty\leq \epsilon/2+ 2\left\| \phi\right\|_\infty \alpha  e^{-\alpha\delta}.
$$
\end{proof}

\begin{proof}[Proof of Theorem \ref{appendix}] As the embedding
	$$
	(\dom(H),\left\| \bullet\right\|_H )\hookrightarrow (\F, \left\| \bullet\right\|_\En  )
	$$
	has a dense image, the second proposition gives that $\F\cap C_0(X)$ is dense in both, $(\dom(H),\left\| \bullet\right\|_H )$ and $(C_0(X),\left\| \bullet\right\|_\infty)$, so that the claim follows from the first proposition.
\end{proof}

\end{document}